\documentclass[11pt]{amsart}

\usepackage{amsmath,amssymb,amsbsy,amsthm,latexsym,
         amsopn,amstext,amsxtra,euscript,amscd,amsthm, mathtools, dsfont}
         
\usepackage{fullpage,amssymb,amsmath}
\usepackage{enumerate}
\usepackage[shortlabels]{enumitem} 
\usepackage{comment}
\usepackage{hyperref} 
\usepackage[capitalise,nameinlink]{cleveref}
\usepackage{url}
\usepackage{xcolor}

\usepackage{bbm}
\usepackage{mathrsfs} 
\newtheorem*{corollary*}{Corollary}
\newtheorem*{theorem*}{Theorem}
\newtheorem{theorem}{Theorem} 

\newtheorem{proposition}[theorem]{Proposition}
\newtheorem{lemma}[theorem]{Lemma}

\theoremstyle{definition}

\newtheorem*{remark}{Remark}

\theoremstyle{remark}

\numberwithin{equation}{section}

\newcommand{\dsum}{\mathop{\sum\sum}}
\newcommand{\pkn}{\mathcal{P}_k(n)}
\newcommand{\pkd}{\mathcal{P}_{k,d}(n)}

\renewcommand{\epsilon}{\varepsilon}

\newcommand{\sk}{S^{(k)}(n)}
\newcommand{\sdk}{S_d^{(k)}(n)}
\newcommand{\sek}{S_e^{(k)}(n)}

\newcommand{\E}{\mathbb{E}}
\newcommand{\p}[2]{\mathcal{P}_{#1}(#2)}
         
\begin{document}

\title{Sums of random multiplicative functions over function fields with few irreducible factors}

\author{Daksh Aggarwal}
\address{
Department of Mathematics \\ 
Grinnell College\\ 
1115 8th Ave \# 3011\\
Grinnell, IA\\
USA \\
50112} 
\email{aggarwal2@grinnell.edu}

\author{Unique Subedi}
\address{
Department of Statistics \\ 
University of Michigan\\ 
1085 University Ave\\
323 West Hall\\
Ann Arbor, MI\\
USA \\
48109
} 
\email{subedi@umich.edu}

\author{William Verreault}
\address{
D\'{e}partement de Math\'{e}matiques et de Statistique\\ 
Universit\'{e} Laval\\ 
Qu\'ebec\\
QC\\
G1V 0A6 \\
Canada} 
\email{william.verreault.2@ulaval.ca}

\author{Asif Zaman}
\address{Department of Mathematics\\
University of Toronto \\
40 St. George Street, Room 6290 \\
Toronto, ON \\
Canada \\
M5S 2E4}
\email{zaman@math.toronto.edu}

\author{Chenghui Zheng}
\address{Department of Statistics\\
University of Toronto \\
100 St.George Street\\
Toronto, ON \\
Canada \\
M5S 3G3}
\email{chenghui.zheng@mail.utoronto.ca}

\maketitle

\vspace*{-4mm}

\begin{abstract}
We establish a normal approximation for the limiting distribution of partial sums of random  Rademacher multiplicative functions over function fields, provided the number of irreducible factors of the polynomials is small enough. This parallels work of Harper for random Rademacher multiplicative functions  over the integers.
\end{abstract}

\section{Introduction} 
Let $\mathcal{M}$ be the set of monic polynomials belonging to the polynomial ring $\mathbb{F}_q[t]$ with coefficients in the finite field $\mathbb{F}_q$ with $q$ elements, where $q \geq 2$ is a prime power.  A \textit{random Rademacher multiplicative function} $f : \mathcal{M} \to \{-1,0,1\}$ over $\mathbb{F}_q[t]$ is obtained by picking independent random variables $f(P)$ uniformly distributed on $\{\pm 1\}$ (that is, taking the value $\pm 1$ with probability $1/2$ each) for monic irreducible polynomials $P$,  extending $f$ multiplicatively to all squarefree monic polynomials, and setting $f$ to be zero for all non-squarefree monic polynomials. For example, if $F=P_1\cdots P_k$ for distinct irreducible monic polynomials $P_1,\dots, P_k$, then $f(F) = f(P_1) \cdots f(P_k)$. For any positive integers $k$ and $n$, set
\[
\p{k}{n} = \{ F \in \mathcal{M} : F \text{ squarefree, } \omega(F) = k, \text{ and } \deg(F) = n\},
\]
where $\omega(F)$ is the number of distinct irreducible factors of the polynomial $F$, and $\deg(F)$ is the degree of $F$. The purpose of this article is to establish the following theorem.

\begin{theorem} \label{thm:Main} 
Let $f$ be a random Rademacher multiplicative function over $\mathbb{F}_q[t]$, where $q \geq 2$ is a fixed prime power.
If $k \geq 1$ satisfies $k=o(\log n)$ as $n\to\infty$, then 
\begin{equation}
\frac{1}{\sqrt{\left|\pkn\right|}}\sum_{F \in \mathcal{P}_k(n)} f(F)	
\label{1.1}
\end{equation}
converges in distribution to the standard normal distribution $N(0,1)$ as $n\to\infty$.
\end{theorem}

This result is motivated by the study of random multiplicative functions over the integers, which were introduced by Wintner \cite{wintner_random_1944} to heuristically model the   M\"{o}bius function.  
A random Rademacher multiplicative function $f : \mathbb{N} \to \{-1,0,1\}$ over the integers is similarly obtained by picking independent random variables $f(p)$ for each prime $p$, extending it multiplicatively to all squarefree integers, and setting it to be zero for all non-squarefree integers. If $\pi_k(x)$ is the number of squarefree integers $\leq x$ with  $k$ distinct prime factors, then the  sum
\begin{equation}
\label{1.2}
\frac{1}{\sqrt{\pi_k(x)}}\sum_{\substack{m \leq x \\ \omega(m)=k} } f(m) 
\end{equation}
parallels the quantity in \eqref{1.1}. Indeed, integers of size $x$ are  known to correspond to polynomials in $\mathbb{F}_q[t]$ of degree $n \approx \log x$.  Improving upon a result of Hough \cite{hough_summation_2011}, Harper \cite{harper_limit_2013} established the following theorem which motivates our \cref{thm:Main}.

\begin{theorem}[Harper] \label{thm:Harper}
Let $f$ be a random Rademacher multiplicative function over the integers. If $k \geq 1$ satisfies $k=o(\log\log  x)$ as $x \to \infty$, then \eqref{1.2}	 converges in distribution to the standard normal $N(0,1)$ as $x \to \infty$. 
\end{theorem}

Notice the range $k=o(\log\log x)$ in \cref{thm:Harper} over the integers corresponds precisely to the range $k = o(\log n)$ in \cref{thm:Main} over the polynomial ring $\mathbb{F}_q[t]$.  \cref{thm:Main} can therefore be viewed as an extension of \cref{thm:Harper} to the function field setting. For an introduction to multiplicative functions over function fields, we refer the reader to work of Granville, Harper, and Soundararajan \cite{granville_mean_2015}, whose conventions we follow here.

 The proof strategy for \cref{thm:Main} adapts Harper's key ideas with the verification of three conditions in a martingale central limit theorem (\cref{Thm:Martingale}). In \cref{sec:PlanProof}, we prepare this strategy and define our martingale difference sequence. The analysis of this martingale allows us to efficiently reduce the theorem to a natural counting problem (\cref{lem:Count}), just as Harper did in Section 4.2 of \cite{harper_limit_2013}. However, this counting problem for function fields introduces cases which did not appear for the integers. The source of these new cases is  simple: two distinct irreducible polynomials can have the same degree, but two distinct rational primes cannot have the same size. Since our martingale is filtered based on the degree of the largest irreducible factor (similar to the size of the prime for integers), this distinction creates new terms in our sums that we must carefully treat; see the remark following \cref{lem:Count} for details. 
 
 In \cref{sec:CompleteProof}, we proceed to analyze these sums and complete the proof of \cref{thm:Main} with some technical estimates.  Although these combinatorial sums are somewhat more intricate, the estimation of these sums is  simpler due to the familiar analytic benefits of function fields over integers. The key technical lemma for this analysis (\cref{lem:KeyLemma}) is proved in \cref{sec:TechnicalProof}. We use recent results on the size of $\p{k}{n}$ by G\'{o}mez-Colunga et al. \cite{gomez2020size} and  Afshar and Porritt \cite{afshar_function_2019}, which respectively parallel classical estimates for $\pi_k(x)$ by Hardy and Ramanujan, and Sathe and Selberg. 

We conclude the introduction with a few remarks on the sharpness of \cref{thm:Main} and possible extensions. Harper  showed that the range $k=o(\log\log x)$ in \cref{thm:Harper} is optimal \cite[Corollary 1]{harper_limit_2013}. He further established a normal approximation for sums like \eqref{1.2} with the looser restriction $\omega(m) \leq k$ and also for a larger class of random multiplicative functions \cite[Theorem 3]{harper_limit_2013}. It would be of interest to determine whether the range $k=o(\log n)$ in \cref{thm:Main} is optimal and whether similar extensions hold in our setting. It seems plausible that such results carry over by similar arguments, but we did not pursue those investigations. If  the range $k=o(\log n)$ is optimal as Harper's work would suggest, then this indicates that the proof of \cref{thm:Main} is quite delicate and sensitive to even minor losses. 

\medskip

\subsection*{Acknowledgments}
\noindent
This research was conducted as part of the 2020 Fields Undergraduate Summer
Research Program. The authors are grateful to the Fields Institute for their financial support and facilitating our online collaboration. 

\section*{Notation} 
Let $q \geq 2$ be a prime power. Let $\mathbb{F}_q[t]$ be the polynomial ring with coefficients in the finite field $\mathbb{F}_q$ with $q$ elements. Let   $\mathcal{M}$ be the set of monic polynomials belonging to $\mathbb{F}_q[t]$.  We shall use capital letters to denote a polynomial $F$ in $\mathcal{M}$,  writing $\deg(F)$ for the degree of the polynomial $F$, $\omega(F)$ for the number of distinct irreducible factors of $F$, and $\mathrm{P}^+(F)$ for the maximum degree of an irreducible dividing $F$. The letters $P$ and $Q$ will be reserved for monic irreducible polynomials. For integers $k, n \geq 1$, let $\mathcal{P}_k(n)$ be the set of squarefree polynomials $F$ in $\mathcal{M}$ with $\omega(F) = k$ and $\deg(F) = n$. The letter $f$ denotes a random Rademacher multiplicative function over $\mathbb{F}_q[t]$. The relation $u \ll v$ means that there exists an absolute positive constant $C$ such that $|u| \leq C v$. If the constant $C$ depends on a parameter, say $\epsilon$, then we shall write $u \ll_{\epsilon} v$.

\section{Plan for the proof of \cref{thm:Main}} 
\label{sec:PlanProof}
For integers $k, n \geq 1$ and a random Rademacher multiplicative function $f$ over $\mathbb{F}_q[t]$, define
\[
S^{(k)}(n) = \sum_{F \in \p{k}{n}} f(F). 
\]
Notice $\E[f(F)] = 0$ for any non-trivial squarefree $F$ because $f$ is multiplicative and $(f(P))_P$ is a sequence of independent random variables with mean zero. Hence, $S^{(k)}(n)$ has mean zero. Also, since  
$\mathbb{E}[f(F)f(G)]=1$ if $F=G$ and $0$ otherwise, it follows that 
\begin{equation}
	\label{eqn:2.1}
	\E\big[ S^{(k)}(n)^2 \big] =\dsum_{F,G\in\p{k}{n}}\mathbb{E}[f(F)f(G)] = |\p{k}{n}|.
\end{equation}
Thus the mean of \eqref{1.1} is zero and its variance is indeed one. 
Our goal is to prove that  $S^{(k)}(n)$,
normalized by its standard deviation $\sqrt{|\p{k}{n}|}$, converges in distribution to the standard normal as $n \to \infty$, provided $k = o(\log n)$. The strategy follows that of Harper \cite{harper_limit_2013} with appropriate modifications and simplifications as mentioned earlier in the introduction.

First, notice $\p{1}{n}$ is the set of  irreducible monic polynomials of degree $n$, so $S^{(1)}(n)$ is a sum of $|\p{1}{n}|$ independent random variables uniform on $\{\pm 1\}$. Thus, the classical central limit theorem implies that $S^{(1)}(n)$ converges in distribution to $N(0,1)$ as $n \to \infty$. We may therefore assume throughout that $k \geq 2$.

\subsection{Central limit theorem for martingale difference sequences}
To prove convergence in distribution to standard normal, we want to use a central limit theorem that gives information on the convergence of the partial sums of a martingale difference sequence. The result we use was obtained by McLeish \cite{mcleish_dependent_1974}, but we state it as it appeared in \cite{harper_limit_2013}.
\begin{theorem}[McLeish] \label{Thm:Martingale}
For $n \in \mathbb{N}$, suppose that $k_n \in \mathbb{N}$, and that $X_{i,n}$, $1 \leq i \leq k_n$, is a martingale difference sequence on $(\Omega, \mathscr{F}, (\mathscr{F}_{i,n})_i, \mathbb{P})$. Write $S_n := \sum_{i \leq k_n}X_{i,n}$ and suppose that the following conditions hold:
\begin{enumerate}[label=\upshape(\roman*)]
\item\label{cond 1}  $\displaystyle\sum_{i \leq k_n}\mathbb{E}[X_{i,n}^2] \to 1$ as $n\to \infty;$
\item\label{cond 2}  for each $\varepsilon > 0$, we have $\displaystyle\sum_{i \leq k_n}\mathbb{E}\left[X_{i,n}^2\mathbf{1}_{|X_{i,n}| > \varepsilon}\right] \to 0$ as $n\to \infty$;
\item\label{cond 3}  $\displaystyle\limsup_{\substack{n\to \infty}} \sum_{i \leq k_n}\sum_{j \leq k_n, j\neq i} \mathbb{E}\left[X_{i,n}^2 X_{j,n}^2\right] \leq 1.$
\end{enumerate}
Then, $S_n$ converges in distribution to $N(0,1)$ as $n\to \infty$.
\end{theorem}
Let us describe the martingale difference sequence in our problem. Let $n \geq k \geq 2$. Write $\mathrm{P}^+(F)$ for the maximum degree of the irreducible factors of $F$. For $d \geq 1$, define
$$
\p{k,d}{n}:=\{F \in \mathcal{P}_k(n) : \mathrm{P}^+(F) = d  \},
$$
and set 
$$
\sdk:= \sum_{F \in \mathcal{P}_{k,d}(n)} f(F).
$$
Notice the set $\p{k,n}{n}$ is empty as $k \geq 2$, so $\p{k}{n}$ is the union of $\p{k,d}{n}$ over $1 \leq d \leq n-1$ and therefore $S^{(k)}(n) = \sum_{d=1}^{n-1} S^{(k)}_d (n)$.  

Writing $F\in \pkd$ as $F=QF'$, where $Q$ is a degree $d$ factor of $F$ (among possibly many), it follows by multiplicativity and independence that
$
\mathbb{E}[f(F)] = \mathbb{E}[f(Q)] \mathbb{E}[f(F')].
$
Since $\E[ f(Q) \mid \{f(P): \deg P< d\} ] = \E[f(Q)] =0$, we get that
\[
\mathbb{E}\left[f(F)\mid \{f(P): \deg P< d\}\right]=0,
\]
and so by the linearity of expectation, it follows that
$$\mathbb{E}\left[\sdk| \{f(P): \deg P < d\}\right] = 0.$$
Hence, if $\mathscr{F}_d$ denotes the sigma algebra generated by $\{f(P): \deg P < d\}$, then $(\sdk)_{d\leq n-1}$ is a martingale difference sequence with respect to $(\mathscr{F}_d)_{d\leq n-1}$.

We will therefore apply \cref{Thm:Martingale} to the random variables $\sdk/\sqrt{\left|\pkn\right|}$, which still form a martingale difference sequence and whose sum over $d\leq n-1$ equals $\sk/\sqrt{\left|\pkn\right|}$, the quantity considered in \cref{thm:Main}. For convenience, we also use the notation 
$$
\p{k,\leq d}{n} := \bigcup_{j\leq d}\p{k,j}{n}.
$$

\subsection{Reduction to some counting problems}
 By a computation similar to \eqref{eqn:2.1}, it follows that $\mathbb{E}[\sdk^2]=\left|\pkd\right|$, so that condition \ref{cond 1} of \cref{Thm:Martingale} holds for all $n$, not just in the limit. 
 Proving \cref{thm:Main} then boils down to verifying conditions \ref{cond 2} and \ref{cond 3} of \cref{Thm:Martingale}. The second condition stated in terms of our normalised random variables asks that for all $\varepsilon>0$,
$$
\sum_{d=1}^{n-1}\mathbb{E}\left[\left(\sdk/\sqrt{\left|\pkn\right|}\right)^2\mathbbm{1}_{\left|\sdk\right|/\sqrt{|\pkn|}>\varepsilon}\right] \to 0
$$
as $n\to\infty.$ This quantity is at most
$$
\varepsilon^{-2}\sum_{d=1}^{n-1}\mathbb{E}\left[\sdk^4/\left|\pkn\right|^2\right].
$$
Thus, it suffices to prove that 
\begin{equation}\label{condition_2}  
\sum_{d=1}^{n-1}\mathbb{E}[\sdk^4]=o(\left|\pkn\right|^2) 
\end{equation}
as $n\to\infty$.
The third condition becomes $$
\limsup_{n\to\infty}\sum_{d=1}^{n-1}\sum_{\substack{e=1 \\ e\neq d}}^{n-1}\mathbb{E}\left[\frac{\sdk^2\sek^2}{\left|\pkn\right|^2}\right]\leq 1.$$
Equivalently, we will show
\begin{equation}\label{condition_3}
   \sum_{d=1}^{n-1}\sum_{\substack{e=1\\ e\neq d}}^{n-1}\mathbb{E}\left[\sdk^2\sek^2\right]\leq(1+o(1))\left|\pkn\right|^2
\end{equation}
as $n\to\infty$. For any $1 \leq d,e \leq n-1$, it will therefore be convenient to express $\mathbb{E}\left[\sdk^2\sek^2\right]$ in terms of an explicit counting problem.
\begin{lemma} \label{lem:Count}
With the same notation as above,
\begin{equation} 
    \mathbb{E}\left[\sdk^2\sek^2\right] \leq |\p{k,d}{n}|  |\p{k,e}{n}|  + I_{k,d,e}(n)+ J_{k,d,e}(n), \nonumber
\end{equation}
where  
\begin{equation} \label{I-sum}
    I_{k,d,e}(n) = \sum_{t=1}^{k-1}\sum_{\ell=1}^{n-1}\sum_{M\in \p{2t, \leq \min\{d,e\}}{2\ell}} 
 \sum_{\substack{A\in \p{t,\leq d}{\ell} \\ A\mid M }}\sum_{ \substack{U\in\p{k-t,\leq d}{n-\ell} \\ \mathrm{P}^+(UA) = d} }
\sum_{\substack{B\in \p{t,\leq e}{\ell} \\ B\mid M}} \sum_{ \substack{V\in\p{k-t,\leq e}{n-\ell} \\ \mathrm{P}^+(VB) = e}  } 1,
\end{equation}
and $J_{k,d,e}(n) = 0$ if $d \neq e$, otherwise 
\begin{equation} \label{J-sum}
J_{k,d,d}(n) = \dsum_{ P,Q \in \p{1}{d} } \sum_{ M'\in \p{2k-2}{2n-2d} }   \dsum_{\substack{A', B'\in \p{k-1}{n-d} \\  A' \mid  \,M', \: B' \mid M'}} 1.
\end{equation}
\end{lemma}

\begin{proof}
Expanding out the sums and applying linearity of expectation, we get that
\begin{equation}
\mathbb{E}\left[\sdk^2\sek^2\right]=\dsum_{W,X\in\p{k,d}{n}}\dsum_{Y,Z\in\p{k,e}{n}}\mathbb{E}\left[f(W)f(X)f(Y)f(Z)\right]. 
\label{eqn:Expand-4tuple}	
\end{equation}
Notice the expectation on the righthand side is nonzero only when $WXYZ$ is a square, in which case it equals $1$. This is the counting problem, which we proceed to reformulate. We may write  $WX$ as the product of a square part $U^2$ and a square-free part $M$ so $W=UA$ and $X=U(M/A)$ for some $A$ that divides $M$.  Note that $U, A,$ and $M/A$ are all relatively prime and the maximum degree of their irreducible factors is $\leq d$. A similar reasoning for $YZ$ gives some other square part, say $V^2$, and forces their squarefree part to be $M$ as well so $Y = VB$ and $Z = V(M/B)$ for some $B$ that divides $M$. Again, $V, B,$ and $M/B$ are all relatively prime and the maximum degree of their irreducible factors is $\leq e$. 

With this notation in mind, we proceed to count the corresponding contributions according to cases. First, if $M = 1$ then $A=B=1$, so this case contributes at most
\[
\sum_{U \in \p{k,d}{n}}\sum_{V \in \p{k,e}{n}} 1 = |\p{k,d}{n}|  |\p{k,e}{n}|. 
\] 
Next, we count the terms in \eqref{eqn:Expand-4tuple} where $M \neq 1$ so $\deg M$ and $\omega(M)$ are both non-zero. As $M$ is non-trivial, we have that $A \in P_{t, \leq d}(\ell)$ for some $t \in \{1,\dots,k\}$ and $\ell \in\{1,\dots,n\}$. Comparing the degrees and number of irreducible factors of  $W = UA$ and $X = U(M/A)$, we deduce that $\mathrm{P}^+(UA) = d$ and
\[
\deg U + \deg A = n, \qquad \deg M = 2 \deg A, \qquad \omega(U) + \omega(A) = k, \qquad \omega(M) = 2\omega(A). 
\]
 As $A \in \p{t,\leq d}{\ell}$, this implies that $U \in \p{k-t, \leq d}{n-\ell}$ and $M \in \p{2t, \leq d}{2\ell}$.   A similar analysis holds when comparing $Y = VB$ and $Z = V(M/B)$ but, since the polynomial $M$ is common to both arguments, it follows that $A$ and $B$ necessarily have the same degree and same number of prime factors and so do $U$ and $V$. Hence, $B \in \p{t,\leq e}{\ell}, V \in \p{k-t,\leq e}{n-\ell}$, and $M \in \p{2t,\leq \min\{d,e\}}{2\ell}$. The terms in \eqref{eqn:Expand-4tuple} with $M \neq 1, t \in \{1,\dots,k-1\}$, and $\ell \in \{1,\dots,n-1\}$ therefore contribute at most $I_{k,d,e}(n)$. 

Continuing with this notation, the last case to consider is when $M \neq 1$ and $t=k$ (or equivalently $\ell =n$) in which case $U=V=1$. Notice $U=V=1$ implies that $WXYZ = U^2 V^2 M^2 = M^2$ has a prime factor of degree $\max\{d,e\}$ yet $\mathrm{P}^+(M) \leq \min\{ d,e\}$. If $d \neq e$, this leads to a contradiction, so this last case occurs if and only if $d=e$. Thus, $M$  has at least two distinct degree $d$ factors in this case and $\mathrm{P}^+(A) = \mathrm{P}^+(B) = d$ where $A$ and $B$ divide $M$.  Thus, there exists a pair of distinct irreducibles $P, Q \in \p{1}{d}$ such that $M=PQM'$, where $M'$ belongs to $\p{2k-2}{2n-2d}$ and at least one of the following holds:
\[
P\mid A \text{ and } Q\mid B, \qquad Q\mid A \text{ and } P\mid B, \qquad P\mid A \text{ and } P\mid B, \qquad Q\mid A \text{ and } Q\mid B. 
\] If, say, the first situation holds then $A = PA'$ and $B = QB'$ for $A',B' \in \p{k-1}{n-d}$ dividing $M'$. A similar statement holds for the other cases.   Combining all of these observations, we see that the terms in \eqref{eqn:Expand-4tuple} with $M \neq 1$ and $t=k$ contribute at most $J_{k,d,e}(n)$, as required. 
\end{proof}
\begin{remark}
    This lemma and its proof possess the key differences between the function field setting and the integers. Crucially, the product $WXYZ$ can form a square in a new way and contribute to \eqref{eqn:Expand-4tuple}.  Namely, if $d \leq e$ then $W$ and $X$ do not need to share the same irreducible factor of degree $d$; these factors of degree $d$ can instead pair with factors from $Y$ and $Z$. This manifests in \eqref{I-sum} by allowing $M$ to have these large irreducible factors of degree $d = \min\{d,e\}$ and also by creating the additional terms \eqref{J-sum} which do not appear in Section 4.2 of \cite{harper_limit_2013}. 
    
    If the irreducible factors of degree $d$ from $W$ and $X$ (resp. of degree $e$ from $Y$ and $Z$) are paired in a one-to-one manner, then only $U$ (resp. $V$) in \eqref{I-sum} would have these large factors of degree $d$ (resp. degree $e$) and moreover \eqref{J-sum} would not exist. This is precisely what happens for Harper in the integer setting. Namely, if integers $w$ and $x$ have largest prime factor $p$, then $p^2$ always divides $wx$ since the size of the prime corresponds uniquely to the prime itself. 
\end{remark}

Now, using \cref{lem:Count} with $d=e$, we see that \eqref{condition_2} becomes the requirement that
\[
\sum_{d=1}^{n-1} \big( |\p{k,d}{n}|^2   + I_{k,d,d}(n) + J_{k,d,d}(n) \big) = o\left(|\p{k}{n}|^2\right). 
\] 
Similarly, \eqref{condition_3} holds provided that
\[
\sum_{d=1}^{n-1} \sum_{e=1}^{n-1} \Big(|\p{k,d}{n}|  |\p{k,e}{n}|  + I_{k,d,e}(n)  \Big) \leq  (1 + o(1))|\p{k}{n}|^2.
\]
Since
\[
\sum_{d=1}^{n-1} \sum_{e=1}^{n-1} |\p{k,d}{n}|  |\p{k,e}{n}|  = |\p{k}{n}|^2,
\]
both  \eqref{condition_2} and  \eqref{condition_3} will therefore be satisfied provided  
\begin{equation}
\label{eqn:ThreeSums}
\sum_{d=1}^{n-1} |\p{k,d}{n}|^2 + \sum_{d=1}^{n-1} \sum_{e=1}^{n-1} I_{k,d,e}(n) + \sum_{d=1}^{n-1} J_{k,d,d}(n) = o( |\p{k}{n}|^2)
\end{equation}
as $n \to \infty$. This establishes \cref{thm:Main} assuming \eqref{eqn:ThreeSums} holds.

%
%
\section{Completing the proof of \texorpdfstring{\cref{thm:Main}}{}}
\label{sec:CompleteProof} 
 
It remains to prove \eqref{eqn:ThreeSums}, which rests on the following key technical lemma whose proof is postponed to \cref{sec:TechnicalProof}.  

\begin{lemma} \label{lem:KeyLemma} 
Fix an integer $r \geq 1$. If $k$ and $n$ are integers such that $r \leq k \leq \frac{1}{3} \log n$, then 
\begin{equation}  \label{iteratedsum}
\mathop{\dsum \cdots \dsum}_{\substack{k_1, n_1, \dots,k_r, n_r \geq 1 \\ k_1+\cdots+ k_r = k \\ n_1+\cdots+ n_r = n} } |\p{k_1}{n_1}|^2 \cdots |\p{k_r}{n_r}|^2 \ll_r \frac{q^{2n} (\log n + 2-\log 2)^{2k-2r}}{n^2 (k-r)!^2}. 
\end{equation}
In particular, if $r \geq 2$ is fixed and $k = o(\log n)$ as $n \to \infty$, then the above is $o(|\p{k}{n}|^2)$. 
\end{lemma}
Assuming \cref{lem:KeyLemma}, it suffices to show that each of the three sums in \eqref{eqn:ThreeSums} are $o( |\p{k}{n}|^2)$ provided $k =o(\log n)$ as $n \to \infty$. We deal with each estimate in separate subsections. 

\subsection{Estimate for $\sum_{d=1}^{n-1}\left|\p{k,d}{n}\right|^2$}  
  For $F\in\p{k,d}{n}$, one has $F=PF'$ for some $P \in \p{1}{d}$ and $F' \in \p{k-1}{n-d}$. This implies that $|\p{k,d}{n}| \leq |\p{1}{d}| |\p{k-1}{n-d}|$ and so
  \begin{equation} 
  \sum_{d=1}^{n-1}\left|\p{k,d}{n}\right|^2 \leq
  \sum_{d=1}^{n-1}|\p{1}{d}|^2\left|\mathcal{P}_{k-1}(n-d)\right|^2. 
  \nonumber
  \end{equation}
This is a subsum of \cref{lem:KeyLemma} with $r=2$ so it is  $o(|\pkn|^2)$ as $n \to \infty$, as required.  

\subsection{Estimate for $\sum_{d=1}^{n-1}\sum_{e=1}^{n-1} I_{k,d,e}(n)$} Consider the definition of $I_{k,d,e}(n)$ in \eqref{I-sum}. The condition $\mathrm{P}^+(UA) = d$ implies that at least one of the following holds: $\mathrm{P}^+(U) = d$ or $\mathrm{P}^+(A) = d$. Summing over $d$ and, in some cases, dropping the requirement that the maximum degree of the irreducible factors of our polynomials is $\leq d$ or $\leq e$, this implies that $\sum_{d=1}^{n-1} I_{k,d,e}(n)$ is at most
\begin{align*}
&  \sum_{t=1}^{k-1}\sum_{\ell=1}^{n-1}\sum_{M\in \p{2t}{2\ell}} 
 \sum_{d=1}^{n-1} \Big(  \sum_{\substack{A\in \p{t,d}{\ell} \\ A\mid M }}\sum_{ \substack{U\in\p{k-t}{n-\ell}  } } +  \sum_{\substack{A\in \p{t}{\ell} \\ A\mid M }}\sum_{ \substack{U\in\p{k-t,d}{n-\ell}  } } \Big)
\sum_{\substack{B\in \p{t,\leq e}{\ell} \\ B\mid M}} \sum_{ \substack{V\in\p{k-t,\leq e}{n-\ell} \\ \mathrm{P}^+(VB) = e}  } 1 \\
& = 2 \sum_{t=1}^{k-1}\sum_{\ell=1}^{n-1}\sum_{M\in \p{2t}{2\ell}} 
    \sum_{\substack{A\in \p{t }{\ell} \\ A\mid M }}\sum_{ \substack{U\in\p{k-t}{n-\ell}  } }  
\sum_{\substack{B\in \p{t,\leq e}{\ell} \\ B\mid M}} \sum_{ \substack{V\in\p{k-t,\leq e}{n-\ell} \\ \mathrm{P}^+(VB) = e}  } 1. 
\end{align*}
Applying the same argument to the condition $\mathrm{P}^+(VB) = e$ and summing over $e$, it follows that $\sum_{d=1}^{n-1} \sum_{e=1}^{n-1} I_{k,d,e}(n) $ is at most
\begin{equation}
\begin{aligned}
 4 \sum_{t=1}^{k-1}\sum_{\ell=1}^{n-1} \Big( \sum_{M\in \p{2t}{2\ell}} 
    \sum_{\substack{A\in \p{t }{\ell} \\ A\mid M }}
\sum_{\substack{B\in \p{t}{\ell} \\ B\mid M}} 1 \Big)  \Big( \sum_{ \substack{U\in\p{k-t}{n-\ell}  } }   \sum_{ \substack{V\in\p{k-t}{n-\ell} }  } 1 \Big). 		
\end{aligned}
\label{eqn:I-simplified}
\end{equation}
Fix $t \in \{1,\dots,k-1\}$ and $\ell \in \{1,\dots,n-1\}$. Notice that the double sum with $U$ and $V$ is equal to $|\p{k-t}{n-\ell}|^2$. Next, consider the triple sum with $M, A,$ and $B$. Writing $G = \gcd(A,B)$, we have that $A = GA', B= GB'$, and $M = GA'B'M'$ for some $M'$ coprime to $A', B',$ and $G$. Since $A$ and $B$ have the same degree and same number of prime factors (and hence so do $A'$ and $B'$), it follows that $G$ and $M'$ must have the same degree and same number of  prime factors. Namely, if $G \in \p{j}{g}$ for some integer $0 \leq j \leq t$ and some integer $0 \leq g \leq \ell$, then $M' \in \p{j}{g}$ and $A', B' \in \p{t-j}{\ell-g}$. Note the case $j=0$ (and hence $g=0$) occurs when $G=M'=1$ so $M = AB$, and the case $j=t$ (and hence $g=\ell$) occurs when $A'=B'=1$ so $M = GM'$.  Combining these observations implies that
\begin{equation}
\begin{aligned}
\sum_{M\in \p{2t}{2\ell}} \sum_{\substack{A\in \p{t }{\ell} \\ A\mid M }} \sum_{\substack{B\in \p{t}{\ell} \\ B\mid M}} 1
& \leq  2|\p{t}{\ell}|^2 +  \sum_{j=1}^{t-1} \sum_{g=1}^{\ell-1} \dsum_{G, M' \in \p{j}{g}} \dsum_{A', B' \in \p{t-j}{\ell-g}} 1 \\
& = 2|\p{t}{\ell}|^2 +  \sum_{j=1}^{t-1} \sum_{g=1}^{\ell-1} |\p{j}{g}|^2 |\p{t-j}{\ell-g}|^2. 
\end{aligned}
\label{eqn:GCD}
\end{equation}
Inserting these estimates in \eqref{eqn:I-simplified}, we conclude that $\sum_{d=1}^{n-1} \sum_{e=1}^{n-1} I_{k,d,e}(n)$ is at most
\begin{align*}
 8 \sum_{t=1}^{k-1} \sum_{\ell=1}^{n-1} |\p{t}{\ell}|^2 |\p{k-t}{n-\ell}|^2  + 4\sum_{t=1}^{k-1} \sum_{\ell=1}^{n-1} \sum_{j=1}^{t-1} \sum_{g=1}^{\ell-1} |\p{k-t}{n-\ell}|^2 |\p{j}{g}|^2 |\p{t-j}{\ell-g}|^2.
\end{align*}
Since $k = o(\log n)$ as $n \to \infty$, both of these sums are $o( |\p{k}{n}|^2)$ by \cref{lem:KeyLemma}, as required. 

\subsection{Estimate for $\sum_{d=1}^{n-1} J_{k,d,d}(n)$}
From \eqref{J-sum}, we have that
\[
\sum_{d=1}^{n-1} J_{k,d,d}(n) =  \sum_{d=1}^{n-1}  |\p{1}{d}|^2  \sum_{ M'\in \p{2k-2}{2n-2d} }   \dsum_{\substack{A', B'\in \p{k-1}{n-d} \\  A' \mid  \,M', \: B' \mid M'}} 1.
\]
Notice the inner triple sum is the same as \eqref{eqn:GCD} with $\ell = n-d$ and $t=k-1$. Thus, $\sum_{d=1}^{n-1} J_{k,d,d}(n)$ is at most
\[
2 \sum_{d=1}^{n-1} |\p{1}{d}|^2  |\p{k-1}{n-d}|^2 + \sum_{d=1}^{n-1}    \sum_{j=1}^{k-2} \sum_{g=1}^{n-d-1} |\p{1}{d}|^2 |\p{j}{g}|^2 |\p{k-j-1}{n-d-g}|^2.
\]
Since $k=o(\log n)$ as $n\to\infty$, all of these sums are $o(|\p{k}{n}|^2)$ by \cref{lem:KeyLemma}. This completes the proof of \eqref{eqn:ThreeSums} and the proof of \cref{thm:Main}. \hfill \qed
\\
\section{Proof of \texorpdfstring{\cref{lem:KeyLemma}}{}} 
\label{sec:TechnicalProof}
All that remains is to prove \cref{lem:KeyLemma}. To do so, we shall first require an estimate for the size of $\p{k}{n}$ that is uniform for all integers $k$ and $n$. G\'{o}mez-Colunga et al. \cite{gomez2020size} have recently established such a result. 

\begin{proposition}[G\'{o}mez-Colunga--Kavaler--McNew--Zhu] \label{prop-HardyRamanujan}
	Uniformly for all $k, n \geq 1$, 
	\[
	|\p{k}{n}| \leq \frac{q^n}{n} \frac{(\log n + 2 - \log 2)^{k-1}}{(k-1)!}.
	\]
\end{proposition}

This corresponds to a classical result of Hardy and Ramanujan \cite{ramanujan1917normal} for the integers: there exists a constant $B>0$ such that, for all $k\geq 1$ and $x\geq 2$, we have
$$
\pi_k(x) \ll \frac{x}{\log x}\frac{(\log\log x+B)^{k-1}}{(k-1)!} ,
$$
where $\pi_k(x)$ is the number of squarefree integers up to $x$ with $k$ prime factors. 

Sathe  \cite{sathe1953problem,sathe1954problem} and Selberg \cite{selberg1954note} famously derived an asymptotic estimate for $\pi_k(x)$ when $k =o(\log \log x)$. We shall also need an asymptotic estimate for $|\p{k}{n}|$ that is valid when $k = o(\log n)$. 
Although the estimate $|\p{k}{n}| \sim \frac{q^n (\log n)^{k-1}}{n (k-1)!}$ (see, e.g.,  \cite{warlimont}) suffices for our purposes, we state here the strongest and most recent result on an estimate for $|\p{k}{n}|$, which is a so-called Sathe--Selberg formula for function fields established by Afshar and Porritt \cite{afshar_function_2019}.

\begin{proposition}[Afshar--Porritt] \label{prop-SatheSelberg}
	Let $A>1$. Uniformly for all $n\geq2$ and $1\leq k\leq A\log n$,
	\[
	  |\p{k}{n}|= \frac{q^n(\log n)^{k-1}}{n(k-1)!}\left(G\left(\frac{k-1}{\log n}\right)+O_A\left(\frac{k}{(\log n)^2}\right)\right),
	  \]
	  where
	  \[
	G(z)=\frac{1}{\Gamma(1+z)}\prod_{\substack{P\in\mathcal{M} \\  P \text{ irreducible} }}\left(1+\frac{z}{q^{\deg P}}\right)\left(1-\frac{1}{q^{\deg P}}\right)^z,
	\]
	and $\Gamma(\cdot)$ is the Gamma function defined as
	$
	\Gamma(z) = \int_{0}^{\infty} x^{z-1}e^{-x} dx.
	$
\end{proposition}

\cref{prop-HardyRamanujan,prop-SatheSelberg} imply our key technical lemma.

\begin{proof}[Proof of \cref{lem:KeyLemma}]
The second estimate follows from  \eqref{iteratedsum} since
\[
  \frac{q^{2n} (\log n + 2-\log 2)^{2k-2r}}{n^2 (k-r)!^2 }   \leq  \frac{q^{2n} (\log n)^{2k-2}}{n^2 (k-1)!^2 } \cdot \frac{ k^{2r-2}}{(\log n)^{2r-2}} \Big(1 + \frac{2-\log 2}{\log n} \Big)^{2k-2r},
\]
and \cref{prop-SatheSelberg} implies that if $k = o(\log n)$ as $n \to \infty$, then $|\p{k}{n}| \sim \frac{q^n (\log n)^{k-1}}{n (k-1)!}$.

To prove \eqref{iteratedsum}, we proceed by induction on $r$. For $r=1$, the claim follows immediately from \cref{prop-HardyRamanujan}. For $r \geq 2$, if $n_1+\cdots+n_r = n$, then at least one of  $n_1,\dots,n_r$  is at most $\lfloor n/r \rfloor $. By symmetry, we may assume it is $n_r$ so the left side of \eqref{iteratedsum} is at most
\[
\ll_r \sum_{k_r=1}^{k-r+1} \sum_{n_r=1}^{\lfloor n/r\rfloor} |\p{k_r}{n_r}|^2 \Big( \mathop{\dsum \cdots \dsum}_{\substack{k_1, n_1, \dots,k_{r-1}, n_{r-1} \geq 1 \\ k_1+\cdots+ k_{r-1} = k-k_r \\ n_1+\cdots+ n_{r-1} = n-n_r} } |\p{k_1}{n_1}|^2 \cdots |\p{k_{r-1}}{n_{r-1}}|^2  \Big).
\]
Notice that $n-n_r \geq n/2$ as $r \geq 2$. Since $k \leq \frac{1}{3}\log n$ by assumption, this implies that $k-k_r \leq k-1 \leq \frac{1}{3}\log(n/2) \leq \frac{1}{3} \log(n-n_r)$. Thus, by the inductive hypothesis, the above is 
\[
   \ll_r \sum_{k_r=1}^{k-r+1}\sum_{n_r=1}^{\lfloor n/r\rfloor} |\p{k_r}{n_r}|^2 \frac{ q^{2(n-n_r)}(\log(n-n_r) + c)^{2k-2k_r-2r+2}}{(n-n_r)^2 (k-k_r-r+1)!^2},
\]
where for brevity we have set $c=2-\log 2$. Applying \cref{prop-HardyRamanujan}, we see that this is at most

\begin{align}
&\sum_{k_r=1}^{k-r+1}\frac{q^{2n}}{(k_r-1)!^2(k-k_r-r+1)!^2} \sum_{n_r=1}^{\lfloor n/r \rfloor}\frac{(\log n_r +c)^{2k_r-2}(\log(n-n_r)+c)^{2k-2k_r-2r+2}}{n_r^2(n-n_r)^2} \nonumber \\
  &\ll_r \frac{q^{2n}}{n^2} \sum_{k_r=1}^{k-r+1}\frac{(\log n +c)^{2k-2k_r-2r+2}}{(k_r-1)!^2(k-k_r-r+1)!^2}  \sum_{n_r=1}^{\lfloor n/r \rfloor}\frac{(\log n_r + c)^{2k_r-2}}{n_r^2}. \label{symmetricsum}
\end{align}
Note that for any integer $m \geq 0$,
$$
\sum_{j=1}^{\infty} \frac{(\log j + c)^m}{j^2} \ll \int_{1}^{\infty} \frac{(\log t +c)^m}{t^2}\; dt = \int_{c}^{\infty} t^{m}e^{c-t}\; dt \ll \int_{0}^{\infty} t^{m}e^{-t}\; dt= m!.
$$
Using this estimate on the inner sum over $n_r$, it follows that \eqref{symmetricsum} is
\begin{equation*}
    \ll_r \frac{q^{2n}}{n^2}\sum_{k_r=1}^{k-r+1}\frac{(2k_r-2)!(\log n +c)^{2k-2k_r-2r+2}}{(k_r-1)!^2(k-k_r-r+1)!^2}.
\end{equation*}
For the final sum over $k_r$, notice that the ratio of consecutive summands is equal to
$$
\frac{2k_r(2k_r-1)}{k_r^2}\frac{(k-k_r-r+1)^2}{(\log n+c)^{2}}
\leq \frac{4k^2}{(\log n)^2} \leq \frac{4}{9}, 
$$ 
since  $k \leq \frac{1}{3} \log n$ by assumption. Hence, the final sum over $k_r$ is dominated by its value at the endpoint $k_r=1$, yielding the desired estimate. This establishes  \cref{lem:KeyLemma}.
\end{proof}

\bibliographystyle{acm}
\bibliography{reference.bib}
\parindent0pt
\end{document}